\documentclass[11pt, reqno]{amsart}

\usepackage{amssymb}
\usepackage{amsmath}
\usepackage{mathrsfs}
\usepackage{amsfonts}
\usepackage{color}
\usepackage{vmargin}
\usepackage{amsthm}
\usepackage{graphicx}

\long\def\symbolfootnote[#1]#2{\begingroup%
\def\thefootnote{\fnsymbol{footnote}}\footnote[#1]{#2}\endgroup}

\setmarginsrb{20mm}{20mm}{20mm}{20mm}{10mm}{10mm}{10mm}{10mm}

\newtheoremstyle{remark}
  {}{}{}{}{\bfseries}{.}{.5em}{{\thmname{#1 }}{\thmnumber{#2}}{\thmnote{ (#3)}}}

\RequirePackage{amsthm}

\newtheorem{tw}{Theorem}[section]

\newtheorem{lem}{Lemma}[section]

\def\vint{\mathop{\mathchoice%
          {\setbox0\hbox{$\displaystyle\intop$}\kern 0.22\wd0%
           \vcenter{\hrule width 0.6\wd0}\kern -0.82\wd0}%
          {\setbox0\hbox{$\textstyle\intop$}\kern 0.2\wd0%
           \vcenter{\hrule width 0.6\wd0}\kern -0.8\wd0}%
          {\setbox0\hbox{$\scriptstyle\intop$}\kern 0.2\wd0%
           \vcenter{\hrule width 0.6\wd0}\kern -0.8\wd0}%
          {\setbox0\hbox{$\scriptscriptstyle\intop$}\kern 0.2\wd0%
           \vcenter{\hrule width 0.6\wd0}\kern -0.8\wd0}}%
          \mathopen{}\int}
\def\={\hspace{-3mm}&=&\hspace{-3mm}}

\renewcommand{\r}{\mathbb{R}}
\newcommand{\n}{\mathbb{N}}

\begin{document}

\date{}

\title{\bf Maximal operator in  H\"older spaces}

\author{Piotr Micha{\l} Bies, Micha{\l} Gaczkowski, Przemys{\l}aw G\'{o}rka \bigskip}

\begin{abstract}
We study the maximal operator on the variable exponent H\"older spaces in the setting of metric measure spaces. The boundedness is proven for 
metric measure spaces satisfying an annular decay property. Let us stress that there are no assumptions on the regularity of the variable 
exponent and the variable exponent can touch values $0$ and $1$. Furthermore, the continuity of the maximal operator between H\"older spaces is investigated. Those results are new even in the Euclidean setting.
\end{abstract}
\maketitle
\bigskip

\medskip


\medskip
\section{Introduction}
The Hardy-Littlewood maximal operator $M$ plays a very important role in the theory of function spaces. The boundedness of $M$ in various types of function spaces is a central issue. It is well known 
that for $1<p \leq \infty$, the maximal operator is bounded on $L^{p}(X, d, \mu)$, where $(X, d, \mu)$ is a doubling metric measure space (see e.g. \cite{Heinonen}). On the other hand, the maximal operator has been also studied in different function spaces, e.g.: Banach function spaces \cite{Sh}, Sobolev spaces \cite{K}, Lebesgue spaces with variable exponent \cite{Diening}, generalized Orlicz spaces \cite{Hasto}. Furthermore, Buckley proved \cite{Bu} that $M$ is bounded in H\"older spaces $C^{0,s}(X)$, where $(X, d, \mu)$ satisfies the $\delta$-annular decay property and $\mu$ is doubling. More recently, G\'{o}rka \cite{Gorka} proved that the maximal operator is bounded in the space of continuous functions $C(X)$, if $(X, d, \mu)$ satisfies the $\delta$-annular property. On the other hand, if no annular decay property is assumed, then $Mf$ can fail to be continuous, even if $f \in C^{0,1}(X)$ (see Example 1.4. in \cite{Bu}). 

The main objective of the paper is to study the maximal operator in the variable exponent H\"older spaces $C^{0,\alpha(\cdot)}(X)$, where $(X, d, \mu)$ satisfies the $\delta$-annular property.\footnote{Variable exponent spaces \cite{diening2011lebesgue} are nowadays used in the description of  non-linear phenomena in fluid mechanics \cite{ruzicka2007electrorheological}, and in image restoration \cite{li2010variable,https://doi.org/10.1049/ipr2.12010}, among other fields.} We shall prove boundedness of maximal operator in $C^{0,\alpha(\cdot)}(X)$. Let us stress that in our result there are no assumptions on the regularity of the variable exponent and the variable exponent can touch values $0$ and $1$. In particular, we do not assume log-H\"{o}lder continuity of the exponent, which in the theory of variable exponent spaces is a commonly used assumption on the exponent.  Moreover, the second main objective is the investigation of the continuity of the maximal operator between $C^{0,\alpha(\cdot)}(X)$ and $C^{0,\beta(\cdot)}(X)$. We prove that $M$ is continuous if $\sup_{x \in X} \beta (x) / \alpha (x) <1$. On the other hand, we show that $M$ is discontinuous on $C^{0,1}(\mathbb{R})$.  This result is rather surprising if we take a look at Luiro's paper \cite{Luiro} about continuity of the maximal operator in Sobolev spaces.
Let us make it clear that those results are new even in the Euclidean setting.

The remainder of the paper is structured as follows. In Section 2, we introduce the notations and recall the definitions. Our first principal assertion, concerning the boundedness of the maximal operator in the variable exponent H\"older spaces, is formulated and proven in Section 3. The last section is devoted to study the continuity of the maximal operator between H\"older spaces.

\section{Preliminaries}
Let $(X, d, \mu)$ be a metric measure space equipped with a metric $d$ and the Borel measure $\mu$. 
We assume that the measure of every open nonempty set is positive and that the measure of every bounded set is finite. 
We shall denote the average of an integrable function $f$ over the measurable set $A$, such that $0<\mu(A)<\infty$, in the following manner
\begin{eqnarray*}
  \vint_{A} f d \mu =\frac{1}{\mu(A)}\int_A f \, d \mu.
\end{eqnarray*}
Let $f :X \rightarrow \mathbb{R}$ be a locally integrable function, then the \texttt{maximal function} $Mf$ is defined as follows
\[
Mf(x) = \sup_{r>0} \vint_{B(x,r)} |f| \, d \mu.
\]
Next, we recall the definition of annular decay property \cite{Bu}. Given $\delta \in (0,1]$, we say that the space $(X, d, \mu)$ satisfies the $\delta$-\texttt{annular decay property} if there 
exists a constant $K\geq 1$ such that for all $x\in X$, $r>0$, $0<\epsilon <1$, we have
\[
  \mu \left(B(x,r) \setminus B(x,r(1-\epsilon))\right)\leq K \epsilon^{\delta}\mu(B(x,r)).
\]

Let $(X,d)$ be a metric space, by $C(X)$ we denote the space of continuous functions on $X$ such that the norm
\[
 \|f\|_{C(X)}= \sup_{x\in X}|f(x)|
\]
is finite. Moreover, for $\alpha :  X \rightarrow [0,1]$ we denote by $C^{0,\alpha(\cdot)}(X)$ the \texttt{variable exponent H\"{o}lder space}, i.e. the space of $f \in C(X)$ such that 
\[
\left\| f  \right\|_{C^{0,\alpha(\cdot)}(X)} := \| f \|_{C(X)}  + \sup_{x\neq y} \frac{|f(x)-f(y)|}{d^{\alpha(x)}(x,y)} < \infty.
\]

\section{Boudedness of the maximal operator}

\begin{tw}\label{t1}
 Suppose that $0<\delta\leq 1$, and that $(X, d, \mu)$ satisfies the $\delta$-annular property. If $\alpha : X \rightarrow [0,\delta]$, then $M: C^{0,\alpha(\cdot)}(X) \rightarrow {C}^{0,\alpha(\cdot)}(X)$ and there exists $C_1>0$
such, that for $f \in C^{0,\alpha(\cdot)}(X)$ the following estimate holds
\[  
\| Mf \|_{{C}^{0,\alpha(\cdot)}(X)} \leq C_1   \| f \|_{C^{0,\alpha(\cdot)}(X)}.
\]
\end{tw}

\begin{proof}
Let us fix $f \in C^{0,\alpha(\cdot)}(X)$ such, that $ \| f \|_{C^{0,\alpha(\cdot)}(X)}  = 1$. By Theorem A from \cite{Gorka} we have that $ Mf \in C(X)$ and the following inequality holds
 \[
  \|Mf\|_{C(X)}\leq \|f\|_{C(X)}.
 \]
Therefore, in order to prove Theorem \ref{t1}, we need to show that for every $x, y \in X$ such that $x\neq y$ 
\begin{equation} \label{sn1}
\frac{|Mf(x)-Mf(y)|}{d^{\alpha(x)}(x,y)}   \leq C_1,  
\end{equation}
where $C_1=\max \left\{7, 1+12K2^{\delta}\right\}$.   
 
Let us fix two distinct points $x,y \in X$ and define $a = d(x,y)$. If $a>1$, then
\[
|Mf(x)- Mf(y)| \leq 2  \| f \|_{C(X)} \leq  2 \| f \|_{C(X)}a^{\alpha(x)},
\]
and \eqref{sn1} holds. Therefore, we can assume that $0<a \leq 1$. 
Let us observe that \eqref{sn1} follows from the following inequality
\begin{equation} \label{eq:basic}
Mf(y) \geq Mf(x) - C_1 \min\left\{ a^{\alpha(x)}, a^{\alpha(y)} \right\}.
\end{equation}
Indeed, if \eqref{eq:basic} holds, then
\[
Mf(x) - Mf(y) \leq C_1 \min\left\{ a^{\alpha(x)}, a^{\alpha(y)} \right\},
\]
and 
\[
Mf(y) - Mf(x) \leq C_1 \min\left\{ a^{\alpha(x)}, a^{\alpha(y)} \right\}.
\]
Therefore, gathering the above inequalities we get \eqref{sn1}. Hence, in order to finish the proof we need to show (\ref{eq:basic}). We shall give the proof of (\ref{eq:basic}) in two cases: $\min\left\{ a^{\alpha(x)}, a^{\alpha(y)} \right\}=a^{\alpha(x)}$, $\min\left\{ a^{\alpha(x)}, a^{\alpha(y)} \right\}=a^{\alpha(y)}$. \\

{\bf Case 1}  $\min\left\{ a^{\alpha(x)}, a^{\alpha(y)} \right\}=a^{\alpha(x)}$. \\
By the very definition of $Mf(x)$, we 
choose $r>0$ such that
\begin{equation} \label{eq:r}
Mf(x) \leq  \vint_{B(x,r)} |f| \, d \mu  + a^{\alpha(x)}.
\end{equation}
We shall consider two subcases.

\emph{\bf Subcase 1.1} $r \leq a$.\\ 
For $z \in B(x,r) \cup B(y,r)$ we have
\begin{eqnarray}\label{est}
|f(x) - f(z)| \leq 2 a^{\alpha(x)}. 
\end{eqnarray}
Indeed, if $z\in B(x,r)$, then
\[
|f(x) - f(z)| \leq d^{\alpha(x)}(x,z) \leq a^{\alpha(x)}\leq 2 a^{\alpha(x)},  
\]
and if $z \in B(y,r)$, then
\[
|f(x)-f(z)| \leq d^{\alpha(x)}(x,z)\leq d^{\alpha(x)}(x,y)+d^{\alpha(x)}(y,z)<a^{\alpha(x)}+r^{\alpha(x)}\leq 2 a^{\alpha(x)}.
\]
Thus, from inequality (\ref{est}) we obtain
\begin{align*}
\vint_{B(x,r)} |f|  \, d \mu -\vint_{B(y,r)} |f| \, d \mu & = \vint_{B(x,r)} \left( |f| - |f(x)| \right)  \, d \mu -\vint_{B(y,r)} \left( |f| -| f(x)| \right)  \, d \mu \\ 
& \leq \vint_{B(x,r)} \left| |f| - |f(x)| \right|  \, d \mu +\vint_{B(y,r)} \left| |f| -| f(x)| \right|  \, d \mu \\
& \leq  \vint_{B(x,r)} \left| f - f(x) \right|  \, d \mu +\vint_{B(y,r)} \left| f - f(x) \right|  \, d \mu \leq 4   a^{\alpha(x)}.
\end{align*}
Next, combining the above inequality with \eqref{eq:r} we get
\[
Mf(y) \geq  \vint_{B(y,r)} |f| \, d \mu \geq \vint_{B(x,r)} |f| \, d \mu - 4  a^{\alpha(x)}  \geq  Mf(x) -  5  a^{\alpha(x)} .
\]
Hence, inequality (\ref{eq:basic}) follows.

\emph{\bf Subcase 1.2} $r > a$.\\
 Let us introduce two sets
\begin{eqnarray*}
S := \left\{g \in L^1(B(x,r+2a)) \, \middle| \,  \vint_{B(x,r)} |g|  \, d \mu -\vint_{B(y,r+a)} |g| \, d \mu \leq 6K2^{\delta}  a^{\alpha(x)}   \right\},
\end{eqnarray*}
and
\begin{eqnarray*}
F:= \left\{g \in L^1(B(x,r+2a)) \, \middle| \,  \|g\|_{L^{\infty}(B(x,r+2a))} \leq A \right\},
\end{eqnarray*}
where $ A=A(x,r) := \min \left\{ 1,2 (3r)^{\alpha(x)}  \right\}$.
 
We shall divide this part of the proof into three steps.

\emph{\bf Step 1.1} If $f \in S$, then (\ref{eq:basic}) holds.\\
Indeed, if $f \in S$, then by the very definition of the maximal function we have
\begin{align*}
Mf(y) &\geq  \vint_{B(y,r + a)} |f| \, d \mu \\
& \geq \vint_{B(x,r)} |f| \, d \mu - 6K2^{\delta}   a^{\alpha(x)} \\
&\geq  Mf(x) -  \left(1 + 6K2^{\delta} \right)   a^{\alpha(x)},
\end{align*}
and \eqref{eq:basic} holds.

\emph{\bf Step 2.1} $F \subset S$.\\
First, let us observe that due to the $\delta$ -annular decay property we have
\begin{eqnarray}\label{dec}
\mu \left( B(x,r+2a) \setminus B(x, r) \right) \leq K \left( \frac{2a}{r+2a} \right)^{\delta} \mu \left( B(x,r+2a) \right) .
\end{eqnarray}
Hence, if $g \in F$ then 
\begin{align*}
\vint_{B(x,r)} |g|  \, d \mu -\vint_{B(y,r+a)} |g| \, d \mu & \leq \frac{1}{\mu(B(x,r))} \int_{B(x,r)} |g|  \, d \mu  -  \frac{1}{\mu(B(y,r+a))}  \int_{B(x,r)} |g| \, d \mu   \\ & \leq \frac{\mu(B(y,r+a)) - \mu(B(x,r))}{\mu(B(x,r))\mu(B(y,r+a))}A \mu(B(x,r))  \\
 & \leq \frac{\mu(B(x,r+2a)) - \mu(B(x,r))}{\mu(B(x,r+2a))}A \leq 2 K \left( \frac{2a}{r+2a} \right)^{\delta} (3r)^{\alpha(x)} \\ 
 & \leq 6 K  \left( \frac{2a}{r+2a} \right)^{\delta} \left(  \frac{r}{a} \right)^{\alpha(x)} a^{\alpha(x)} \leq  6 K  \left( \frac{2a}{r+2a} \right)^{\delta} \left(  \frac{r}{a} \right)^{\delta} a^{\alpha(x)} \\
 & =   6  K \left( \frac{2ar}{ar+2a^2} \right)^{\delta}  a^{\alpha(x)} \leq 6 K  2^{\delta} a^{\alpha(x)}.
\end{align*}
Thus, we get $g \in S$.

\emph{\bf Step 3.1}  $f \in S$.\\
If $\alpha(x)=0$ or $r\geq \frac{1}{3}2^{-1/\alpha(x)}$ with $\alpha(x)>0$, then $A = 1$ and since $\| f \|_{C^{0,\alpha(\cdot)}(X)}  = 1$, we have $f \in F$. Hence, by Step 2.1 we get $f \in S$. 

We are left with the case $\alpha(x)>0$ with $r \in \left( a, \frac{1}{3}2^{-1/\alpha(x)} \right)$. 
Let us introduce two auxiliary functions:
\[
f_1= f - \sup_{z \in B(x,r+2a)} f(z),
\]
\[
f_2 = f - \inf_{z \in B(x,r+2a)} f(z).
\]
We claim that $f_1, f_2 \in S$. Indeed, for any $\epsilon > 0$ we find $z_1, z_2 \in B(x,r+2a)$ such, that $f_1(z_1) \geq - \epsilon$ and $f_2(z_2) \leq \epsilon$. Then, for  $z \in B(x,r+2a)$
\begin{align*}
|f_1(z)| &\leq |f_1(z)-f_1(x)|+ |f_1(x)-f_1(z_1)|  + |f_1(z_1)|\\
& \leq  d(z,x)^{\alpha(x)} + d(z_1,x)^{\alpha(x)} + \epsilon \\ 
& \leq 2 \left(  r+ 2a  \right)^{\alpha(x)} + \epsilon  \\
&\leq  2\left(  3 r  \right)^{\alpha(x)}  + \epsilon, 
\end{align*}
and in the same fashion we have 
\begin{align*}
|f_2(z)| &\leq |f_2(z)-f_2(x)|+ |f_2(x)-f_2(z_2)|  + |f_2(z_2)| \leq  d(z,x)^{\alpha(x)} + d(z_2,x)^{\alpha(x)} + \epsilon \\ 
& \leq 2 \left(  r+ 2a  \right)^{\alpha(x)} + \epsilon  \leq  2\left(  3 r  \right)^{\alpha(x)}  + \epsilon.
\end{align*}
Hence,
\begin{equation*}
\left \| f_1  \right\|_{ L^{\infty}(B(x,r+2a)) } \leq 2 \left(  3 r  \right)^{\alpha(x)} + \epsilon,\\
\left \| f_2  \right\|_{ L^{\infty}(B(x,r+2a)) } \leq 2 \left(  3 r  \right)^{\alpha(x)} + \epsilon.
\end{equation*}
Thus, we pass to the limit $\epsilon \rightarrow 0^+$ and since $A(x,r)=2 \left(  3 r  \right)^{\alpha(x)}$ we get $f_1, f_2 \in F \subset S$.

Let us observe that for $z \in B(x,r+2a)$ 
\[
-A \leq f_1(z) \leq 0 \quad \text{and} \quad 0 \leq f_2(z) \leq A.
\]

Therefore, if $f(z_0) \geq A$ for some $z_0 \in B(x,r+2a)$, then for any $z \in B(x,r+2a)$ 
\[
f(z) = f(z_0) + f(z) - f(z_0) \geq A + f_1(z) \geq A-A = 0.
\]
Hence, since $f_2 \in S$, we get
\begin{align*}
\vint_{B(x,r)} |f|  \, d \mu -\vint_{B(y,r+a)} |f| \, d \mu &= \vint_{B(x,r)} f  \, d \mu -\vint_{B(y,r+a)} f \, d \mu \\
&= \vint_{B(x,r)} f_2  \, d \mu -\vint_{B(y,r+a)} f_2 \, d \mu \\
& = \vint_{B(x,r)} |f_2|  \, d \mu -\vint_{B(y,r+a)} |f_2| \, d \mu \\
&\leq 6K2^{\delta}  a^{\alpha(x)},
\end{align*}
and we have $f \in S$.

Similarly, if $f(z_0) \leq -A$ for some $z_0 \in B(x,r+2a)$, then for any $z \in B(x,r+2a)$ 
\[
f(z)\leq 0,
\]
and we obtain
\begin{align*}
\vint_{B(x,r)} |f|  \, d \mu -\vint_{B(y,r+a)} |f| \, d \mu &= -\vint_{B(x,r)} f  \, d \mu +\vint_{B(y,r+a)} f \, d \mu \\
& \leq 6K2^{\delta}  a^{\alpha(x)}.
\end{align*}
Finally, if for every $z \in B(x,r+2a)$ we have $-A \leq f(z) \leq A$, then $f \in F$ and therefore $f \in S$.\\

{\bf Case 2} $\min\left\{ a^{\alpha(x)}, a^{\alpha(y)} \right\}=a^{\alpha(y)}$. \\
The structure of the proof, in this case, is similar to Case 1. Nevertheless, for the convenience of the reader and clarity of the proof, we give the full argumentation. 
Let us fix $r>0$ such that
\begin{equation} \label{eq:r1}
Mf(x) \leq  \vint_{B(x,r)} |f| \, d \mu  + a^{\alpha(y)}.
\end{equation}

\emph{\bf Subcase 2.1} $r \leq a$.\\ 
For $z \in B(x,r) \cup B(y,r)$ we have
\begin{eqnarray}\label{est1}
|f(x) - f(z)| \leq 3 a^{\alpha(y)}. 
\end{eqnarray}
Indeed, if $z\in B(x,r)$, then
\begin{align*}
|f(x) - f(z)| &\leq |f(x)-f(y)|+|f(y)-f(z)|\leq d^{\alpha(y)}(x,y)+d^{\alpha(y)}(y,z)\\
&\leq 2 d^{\alpha(y)}(x,y)+d^{\alpha(y)}(x,z)< 2a^{\alpha(y)}+r^{\alpha(y)}\leq 3a^{\alpha(y)},  
\end{align*}
and if $z \in B(y,r)$, then
\[
|f(x)-f(z)|\leq |f(x)-f(y)|+|f(y)-f(z)| \leq d^{\alpha(y)}(x,y)+d^{\alpha(y)}(y,z)<2a^{\alpha(y)}.
\]
Next, by inequality (\ref{est1}) we get
\begin{align*}
\vint_{B(x,r)} |f|  \, d \mu -\vint_{B(y,r)} |f| \, d \mu &  \leq  6 a^{\alpha(y)},
\end{align*}
and gathering the above inequality with \eqref{eq:r1} we have
\[
Mf(y) \geq  \vint_{B(y,r)} |f| \, d \mu \geq \vint_{B(x,r)} |f| \, d \mu - 6   a^{\alpha(y)}  \geq  Mf(x) -  7  a^{\alpha(y)}.
\]
Therefore, inequality (\ref{eq:basic}) follows.

\emph{\bf Subcase 2.2} $r > a$.\\
We introduce two sets
\begin{eqnarray*}
\tilde{S} := \left\{g \in L^1(B(x,r+2a)) \, \middle| \,  \vint_{B(x,r)} |g|  \, d \mu -\vint_{B(y,r+a)} |g| \, d \mu \leq 12K2^{\delta}  a^{\alpha(y)}   \right\},
\end{eqnarray*}
and
\begin{eqnarray*}
\tilde{F}:= \left\{g \in L^1(B(x,r+2a)) \, \middle| \,  \|g\|_{L^{\infty}(B(x,r+2a))} \leq A \right\},
\end{eqnarray*}
where $ A=A(y,r) := \min \left\{ 1,4 (3r)^{\alpha(y)}  \right\}$.

\emph{\bf Step 1.2} If $f \in \tilde{S}$, then (\ref{eq:basic}) holds.\\
Indeed, if $f \in \tilde{S}$, then 
\begin{align*}
Mf(y) &\geq   \vint_{B(x,r)} |f| \, d \mu - 12K2^{\delta}   a^{\alpha(y)} \\
&\geq  Mf(x) -  \left(1 + 12K2^{\delta} \right)   a^{\alpha(y)}.
\end{align*}

\emph{\bf Step 2.2} $\tilde{F} \subset \tilde{S}$.\\
Let $g \in \tilde{F}$ then, inequality (\ref{dec}) yields
\begin{align*}
\vint_{B(x,r)} |g|  \, d \mu -\vint_{B(y,r+a)} |g| \, d \mu  & \leq 4 K \left( \frac{2a}{r+2a} \right)^{\delta} (3r)^{\alpha(y)} \\ 
 & \leq  12 K  \left( \frac{2a}{r+2a} \right)^{\delta} \left(  \frac{r}{a} \right)^{\delta} a^{\alpha(y)} \\
 & \leq 12 K  2^{\delta} a^{\alpha(y)}.
\end{align*}

\emph{\bf Step 3.2}  $f \in \tilde{S}$.\\
If $\alpha(y)=0$ or $r\geq \frac{1}{3}4^{-1/\alpha(y)}$ with $\alpha(y)>0$, then $f \in \tilde{F}$ and by Step 2.2 we get $f \in \tilde{S}$.

Hence, we need to consider the case $\alpha(y)>0$ with $r \in \left( a, \frac{1}{3}4^{-1/\alpha(y)} \right)$. 
We shall use auxiliary functions $f_1, f_2$ defined in Step 3.1. Let us observe that $f_1, f_2 \in \tilde{S}$. For a fixed $\epsilon > 0$ we choose $z_1, z_2 \in B(x,r+2a))$ such, that $f_1(z_1) \geq - \epsilon$ and $f_2(z_2) \leq \epsilon$. Thus, for  $z \in B(x,r+2a)$
\begin{align*}
|f_1(z)| &\leq |f_1(z)-f_1(y)|+ |f_1(y)-f_1(z_1)|  + |f_1(z_1)|\\
& \leq  d(z,y)^{\alpha(y)} + d(z_1,y)^{\alpha(y)} + \epsilon \leq 2 d^{\alpha(y)}(x,y)+d^{\alpha(y)}(x,z)+d^{\alpha(y)}(x,z_1)+\epsilon\\ 
& \leq 2 \left(  r+ 2a  \right)^{\alpha(y)}+ 2a^{\alpha(y)}+ \epsilon  
\leq  4\left(  3 r  \right)^{\alpha(y)}  + \epsilon, 
\end{align*}
and 
\begin{align*}
|f_2(z)| &\leq |f_2(z)-f_2(y)|+ |f_2(y)-f_2(z_2)|  + |f_2(z_2)| \leq  d(z,y)^{\alpha(y)} + d(z_2,y)^{\alpha(y)} + \epsilon \\ 
&  \leq  4\left(  3 r  \right)^{\alpha(y)}  + \epsilon.
\end{align*}
Therefore,
\begin{equation*}
\left \| f_1  \right\|_{ L^{\infty}(B(x,r+2a)) } \leq 4 \left(  3 r  \right)^{\alpha(y)},\\
\left \| f_2  \right\|_{ L^{\infty}(B(x,r+2a)) } \leq 4 \left(  3 r  \right)^{\alpha(y)},
\end{equation*}
and since $A(y,r)=4 \left(  3 r  \right)^{\alpha(y)}$, we get $f_1, f_2 \in \tilde{F} \subset \tilde{S}$.

Finally, since $f_1, f_2 \in  \tilde{S}$ and $-A \leq f_1(z) \leq 0  \leq f_2(z) \leq A$ for $z \in B(x,r+2a)$, we get that  $f \in \tilde{S}$. This finishes the proof of the theorem.
\end{proof}

\section{Continuity of the maximal operator}
\begin{tw}
Let $\delta \in (0, 1]$ and $(X, d, \mu)$ satisfies the $\delta$-annular property. If $\alpha : X \rightarrow (0,1]$ and $\beta : X \rightarrow [0,1]$ satisfy $\sup_{x \in X} \beta(x)/\alpha(x)<1$, then the operator
$$
M:C^{0,\alpha(\cdot)}(X)\rightarrow C^{0,\beta(\cdot)}(X)
$$
is continuous.
\end{tw}
\begin{proof}
Let us note since $\beta(\cdot) \leq \alpha(\cdot)$, we have $Id : C^{0,\alpha(\cdot)}(X)\rightarrow C^{0,\beta(\cdot)}(X)$. Therefore, due to Theorem \ref{t1} we get $M:C^{0,\alpha(\cdot)}(X)\rightarrow C^{0,\beta(\cdot)}(X)$ is bounded. 

In order to prove the continuity of $M$ we fix $f\in C^{0,\alpha(\cdot)}(X)$ and sequence $\{f_n\}\subset C^{0,\alpha(\cdot)}(X)$ such that $f_n\rightarrow f$ in $C^{0,\alpha(\cdot)}(X)$. It is easy to see that $Mf_n\to Mf$ in $C(X)$. Thus, it is left to show that 
$$
\sup_{x\neq y}\frac{|Mf_n(x)-Mf(x)-Mf_n(y)+Mf(y)|}{d(x,y)^{\beta(x)}}\rightarrow 0.
$$
By Theorem \ref{t1} we know that sequence $\{Mf_n\}$ is bounded in $C^{0,\alpha(\cdot)}(X)$, so we can assume that there exists $N\geq1$ such that $\|Mf_n\|_{C^{0,\alpha(\cdot)}(X)}\leq N$ for all $n$ and $\|Mf\|_{C^{0,\alpha(\cdot)}(X)}\leq N$. Furthermore, we can assume that $\|f_n-f\|_{C^{0,\alpha(\cdot)}(X}<1/2$ for large $n$.

Therefore, for $x, y \in X$ such that $x \neq y$, we get the following string of inequalities
\begin{align*}
&\frac{|Mf_n(x)-Mf(x)-Mf_n(y)+Mf(y)|}{d(x,y)^{\beta(x)}}\\
&=\left(\frac{|Mf_n(x)-Mf(x)-Mf_n(y)+Mf(y)|}{d(x,y)^{\alpha(x)}}\right)^{\frac{\beta(x)}{\alpha(x)}}|Mf_n(x)-Mf(x)-Mf_n(y)+Mf(y)|^{1-\frac{\beta(x)}{\alpha(x)}}\\
&\leq \left(\frac{|Mf_n(x)-Mf_n(y)|}{d(x,y)^{\alpha(x)}}+\frac{|Mf(x)-Mf(y)|}{d(x,y)^{\alpha(x)}}\right)^{\frac{\beta(x)}{\alpha(x)}}\left(2\|Mf_n-Mf\|_{C(X)}\right)^{1-\frac{\beta(x)}{\alpha(x)}}\\
&\leq \left(2N\right)^{\left(\frac{\beta}{\alpha}\right)^+}\left(2\|Mf_n-Mf\|_{C(X)}\right)^{1-\left(\frac{\beta}{\alpha}\right)^+},
\end{align*}
where $\left(\frac{\beta}{\alpha}\right)^+=\sup_{x \in X} \frac{\beta(x)}{\alpha(x)}$. 
Hence,
$$
\sup_{x\neq y}\frac{|Mf_n(x)-Mf(x)-Mf_n(y)+Mf(y)|}{d(x,y)^{\beta(x)}}\leq \left(2N\right)^{\left(\frac{\beta}{\alpha}\right)^+}\left(2\|Mf_n-Mf\|_{C(X)}\right)^{1-\left(\frac{\beta}{\alpha}\right)^+}.
$$
Since the right-hand side of the above inequality goes to $0$ when $n\rightarrow\infty$, the proof follows.
\end{proof}
\begin{tw}\label{prz}
There exist $f, f_n \in C^{0,1}(\r)$ such that $f_n\rightarrow f$ in $C^{0,1}(\r)$ and $Mf_n\not\rightarrow Mf$ in $C^{0,1}(\r)$.
\end{tw}
\begin{proof}
First of all we shall prove the following lemma.
\begin{lem} \label{per}
Let $T >0$ and let $f \in C(\r,\r)$ be a $T-$periodic function. Then, for any $x\in\r$ there exists $r\in [0,T]$ such that 
$$
Mf(x)=\vint_{x-r}^{x+r}|f(t)|\, dt.\footnote{We define $\vint_{x-r}^{x+r}|f(t)|\, dt$ as $|f(x)|$ for $r=0$.}
$$
\end{lem}
\begin{proof}
Let us observe that for a fixed $x \in \r$, the map $f_x: \r \rightarrow \r$ defined as $f_x (t)=f(x+t)$ is $T-$periodic and $Mf_x(0)=Mf(x)$. Therefore, it is enough to prove the lemma for $x=0$. Let us denote 
$$a=\max_{0\leq r\leq T}\vint_{-r}^r|f(t)|\,dt.
$$
We shall show 
$$
\vint_{-r}^r|f(t)|\,dt\leq a
$$
for all $r\geq 0$. The inequality is obvious for $0\leq r\leq T$. We shall prove it for $r>T$. We can write $r$ as $nT+b$, where $n\in\n$ and $0\leq b<T$. Hence,
\begin{align*}
\vint_{-r}^r|f(t)|\,dt&=\frac{\int_{-b}^b|f(t)|\,dt+\sum_{k=1}^n\left(\int_{b+(k-1)T}^{b+kT}|f(t)|\, dt+\int^{-b-(k-1)T}_{-b-kT}|f(t)|\, dt\right)}{2nT+2b}
\\&=\frac{\int_{-b}^b|f(t)|\,dt+2n\int_{-T/2}^{T/2}|f(t)|\, dt}{2nT+2b}=\frac{2b\vint_{-b}^b|f(t)|\,dt+2nT\vint_{-T/2}^{T/2}|f(t)|\, dt}{2nT+2b}\\
&\leq\frac{2ba+2nTa}{2nT+2b}=a,
\end{align*}
and the proof is complete.
\end{proof} 
Now, we are in position to prove Theorem \ref{prz}. Let $f:\r\rightarrow\r$ be a continuous and $2-$periodic function such that $f(x)=|x|$ for $x\in [-1,1]$. Next, we define a sequence $f_n(x)=f(x)-\frac{1}{n}$ for $x\in\r$. It is easy to see that $f_n\rightarrow f$ in $C^{0,1}(\r)$. 

We shall show $Mf_n\not\rightarrow Mf$ in $C^{0,1}(\r)$. It is obvious that $Mf$ is $2-$periodic and even function. Thus, it is enough to recognize $Mf$ on $[0,1]$.
By the straightforward integration, for $x\in [0,\frac 1 2]$ we have 
$$
\vint_{x-r}^{x+r}|f(t)|\,dt=\begin{cases}
x,&\textrm{for }r\in [0,x]\\
\frac{x^2+r^2}{2r},&\textrm{for }r\in [x,1-x]\\
1-x+\frac{2x-1}{2r},&\textrm{for }r\in[1-x,1+x]\\
2-\frac{r}{2}-\frac{x^2+2}{2r},&\textrm{for }r\in [1+x,2-x]\\
\frac{(r-2) x}{r}+\frac{1}{r},&\textrm{for }r\in[2-x,2],
\end{cases}
$$
and for $x\in [\frac{1}{2},1]$ we get 
$$
\vint_{x-r}^{x+r}|f(t)|\,dt =\begin{cases}
x,&\textrm{for }r\in[0,1-x]\\
\frac{-x^2+2 x-1}{2 r}-\frac{r}{2}+1,&\textrm{for }r\in[1-x,x]\\
\frac{x}{r}-\frac{1}{2 r}-x+1,&\textrm{for }r\in[x,2-x]\\
\frac{x^2-2 x+3}{2 r}+\frac{r}{2}-1,&\textrm{for }r\in [2-x,1+x]\\
\frac{1-2 x}{r}+x,&\textrm{for }r\in [1+x,2].
\end{cases}
$$
Therefore, having in mind Lemma \ref{per}, we get
$$
Mf(x)=\begin{cases}
 2-\sqrt{x^2+2}, & \textrm{ for }0\leq x\leq \frac{1}{2} \\
 x, & \textrm{ for }\frac{1}{2}<x\leq 1.
 \end{cases}
$$
Next, for large $n$, we have
$$
\vint_{\frac 1 2-r}^{\frac 1 2+r}|f_n(t)|\,dt=\begin{cases}
\frac 1 2-\frac{1}{n},&\textrm{for }r\in[0,\frac 1 2-\frac{1}{n}]\\
\frac{(n-2)^2}{8 n^2 r}+\frac{r}{2},&\textrm{for }r\in[\frac 1 2-\frac{1}{n},\frac 1 2]\\
\frac{-n^2-4 n+4}{8 n^2 r}-\frac{r}{2}+1,&\textrm{for }r\in[\frac 1 2,\frac 1 2+\frac 1 n]\\
\frac{1}{n^2 r}+\frac{n-2}{2 n},&\textrm{for }r\in[\frac 1 2+\frac 1 n,\frac32-\frac1n]\\
\frac{3 \left(3 n^2-4 n+4\right)}{8 n^2 r}+\frac{r}{2}-1,&\textrm{for }r\in[\frac32-\frac1n,\frac32]\\
-\frac{3 \left(3 n^2+4 n-4\right)}{8 n^2 r}-\frac{r}{2}+2,&\textrm{for }r\in[\frac32,\frac32+\frac1n]\\
\frac{2}{n^2 r}+\frac{n-2}{2 n},&\textrm{for }r\in[\frac32+\frac1n,2].
\end{cases}
$$
Thus, by Lemma \ref{per}, we get
$$
Mf_n(1/2)=1-\sqrt{-\frac{1}{n^2}+\frac{1}{n}+\frac{1}{4}}.
$$
Furthermore, let us define $d_n=\frac{1}{2}-\frac{1}{4 n^2}$. Again, basic integration gives us
$$
\vint_{d_n-r}^{d_n+r}|f_n(t)|\,dt=\begin{cases}
d_n-\frac 1 n,&\textrm{for }r\in[0,d_n-\frac{1}{n}]\\
\frac{4 n^4-16 n^3+12 n^2+8 n+1}{32 n^4 r}+\frac{r}{2},&\textrm{for }r\in [d_n-\frac 1 n,d_n]\\
\frac{2 n^2-1}{4 n^2}+\frac{-2 n^2+2 n+1}{4 n^3 r},&\textrm{for }r\in [d_n,1-d_n]\\
\frac{-4 n^4-16 n^3+12 n^2+8 n-1}{32 n^4 r}-\frac{r}{2}+1,&\textrm{for } r\in[1-d_n,d_n+\frac1n]\\
\frac{3}{4 n^2 r}+\frac{2 n^2-4 n+1}{4 n^2},&\textrm{for }r\in [d_n+\frac{1}{n},2-\frac{1}{n}-d_n]\\
\frac{36 n^4-48 n^3+52 n^2-8 n+1}{32 n^4 r}+\frac{r}{2}-1,&\textrm{for }r\in[2-\frac 1 n-d_n,d_n+1]\\
\frac{2 n^2-1}{4 n^2}+\frac{-6 n^2+8 n-1}{4 n^3 r},&\textrm{for }r\in[d_n+1,2-d_n],\\
\frac{-36 n^4-48 n^3+52 n^2-8 n-1}{32 n^4 r}-\frac{r}{2}+2,&\textrm{for }r\in [2-d_n,2+\frac 1 n-d_n] \\
\frac{5}{2 n^2 r}+\frac{2 n^2-4 n-1}{4 n^2},&\textrm{for }r\in [2+\frac 1 n-d_n,2].
\end{cases}
$$
Hence, by an elementary considerations, for large $n$ we get
$$
\vint_{d_n-r}^{d_n+r}|f_n(y)|\,dy\leq Mf_n(1/2)
$$ 
for $r\in[0,2]$. 
Hence, the above inequality and Lemma \ref{per} yield 
\begin{eqnarray}\label{ine}
Mf_n(d_n)\leq Mf_n(1/2)
\end{eqnarray}
for large $n$.

Finally,  using the expression for $Mf$, we easily get 
$$
\frac{Mf(d_n)-Mf(\frac 1 2)}{\frac{1}{2}-d_n}\rightarrow\frac{1}{3}.
$$
Therefore, gathering (\ref{ine}) with the above convergence, for sufficiently big $n$ we have
\begin{align*}
\|Mf_n -Mf\|_{C^{0,1}(\r)} &\geq \frac{|Mf_n(\frac{1}{2})-Mf(\frac{1}{2})-Mf_n(d_n)+Mf(d_n)|}{|\frac{1}{2}-d_n|}\\
&=\frac{Mf(d_n)-Mf(\frac 1 2)}{\frac{1}{2}-d_n}+\frac{Mf_n(\frac{1}{2})-Mf_n(d_n)}{\frac{1}{2}-d_n}\\
&\geq \frac{Mf(d_n)-Mf(\frac 1 2)}{\frac{1}{2}-d_n}\geq \frac{1}{6}.
\end{align*}
This proves that $Mf_n\not\rightarrow Mf$ in $C^{0,1}(\r)$.
\end{proof}

\subsection*{Acknowledgement} The research of the last author was funded by (POB Cybersecurity and data analysis) of Warsaw University of Technology within the Excellence Initiative: Research University (IDUB) programme.

\bibliographystyle{plain}
\bibliography{HLMaximalFunction}

\smallskip
{\small Piotr Micha{\l} Bies}\\
\small{Department of Mathematics and Information Sciences,}\\
\small{Warsaw University of Technology,}\\
\small{Pl. Politechniki 1, 00-661 Warsaw, Poland} \\
{\tt biesp@mini.pw.edu.pl}\\
\smallskip
{\small Micha{\l} Gaczkowski}\\
\small{Department of Mathematics and Information Sciences,}\\
\small{Warsaw University of Technology,}\\
\small{Pl. Politechniki 1, 00-661 Warsaw, Poland} \\
{\tt M.Gaczkowski@mini.pw.edu.pl}\\
\smallskip
{\small Przemys{\l}aw  G\'orka}\\
\small{Department of Mathematics and Information Sciences,}\\
\small{Warsaw University of Technology,}\\
\small{Pl. Politechniki 1, 00-661 Warsaw, Poland} \\
{\tt  pgorka@mini.pw.edu.pl }\\

\end{document}